\documentclass[11pt]{article}
\usepackage[T1]{fontenc}
\usepackage{amsfonts,amsmath}
\usepackage{amssymb}
\usepackage{amsthm}
\usepackage{hyperref}

\newtheorem{theorem}{Theorem}
\newtheorem{lemma}[theorem]{Lemma}

\newtheorem{cor}[theorem]{Corollary}
\newtheorem{prop}[theorem]{Proposition}

\newtheorem{prb}{Open problem}

\def\vertex(#1){\put(#1){\circle*{2}}}
\def\vertexo(#1){\put(#1){\circle{2}}}
\def\vert(#1){\put(#1){\circle*{1.5}}}
\def\verto(#1){\put(#1){\circle{1.5}}}
\def\lab(#1)#2{\put(#1){\makebox(0,0)[c]{#2}}}
\title{An improved lower and upper bound of the $k$-limited domination number}

\author{
Gordana Radi\'{c} \\
\small \it University of Maribor, FERI, Koroška cesta 46, 2000 Maribor, Slovenia \\
\small \it IMFM, Jadranska ulica 19, 1000 Ljubljana, Slovenia \\
\small \tt gordana.radic@um.si
}

\date{}    

\begin{document}
\maketitle

\begin{abstract}
We continue the study of $k$-limited domination in graphs, a domination variant in which each vertex of a dominating set may dominate at most $k$ vertices outside the set. This concept extends classical domination by incorporating capacity constraints on dominating vertices.

We improve the general lower bound for the $k$-limited domination number $\gamma_k^L(G)$ by employing the concept of $k$-capacitated domination. As a consequence, we characterize graphs satisfying $\gamma_k^L(G)=\lceil \frac{n}{k+1} \rceil$. We also refine the known upper bound for graphs with $k<\delta(G)$ using the $k$-limited packing number. Under this condition, we describe graphs attaining $\gamma_k^L(G)=n-k$. These results extend and unify previous investigations for the case $k=1$ and provide a complete characterization of graphs attaining the extreme values of the $k$-limited domination number under the considered assumptions.
\end{abstract}


\section{Introduction}

Let $G=(V(G),E(G))$ be a finite simple graph. A set $D\subseteq V(G)$ is called a \emph{dominating set} of $G$ if every vertex in $V(G)\setminus D$ has a neighbor in $D$; the minimum cardinality among all dominating sets of $G$ is called the \emph{domination number} of $G$, and it is denoted by $\gamma(G)$. Determining the domination number of a graph, usually referred to as the \emph{domination problem}, is one of the most extensively studied topics in Graph Theory, both from theoretical and applied perspectives.

The significance of domination stems largely from its natural interpretation in real-world systems, where it models the problem of placing the minimum number of resources in a network so that every location is either itself a resource or adjacent to one. Motivated by such applications, numerous variants have been introduced to incorporate additional restrictions that arise in practical settings. Consequently, domination-related problems have found applications across a wide range of areas, including computer science, telecommunications, logistics, biology, chemistry, and medicine. For a comprehensive treatment of domination theory, including historical background, applications, and further references, we refer the reader to the monographs \cite{HHH1, HHH3, HHH2}.

Among the numerous domination variants surveyed in these monographs, many impose additional restrictions on vertices outside the dominating set (vertices in the dominating set are sometimes called dominators). Typical examples require every non-dominating vertex to be adjacent to several dominators or prescribe bounds on the number of dominators that may dominate a given vertex. Recently, in \cite{NAS}, a different perspective was introduced through the concept of \emph{$k$-limited domination}, where the restriction is imposed on the vertices inside the dominating set. This approach is motivated by situations in which the selected vertices represent facilities, communication nodes, servers, or service providers with limited operating capacity. In such settings, a single resource cannot usually serve arbitrarily many neighboring units, making it natural to require that every dominating vertex is responsible for at most a bounded number of vertices outside the dominating set. Consequently, $k$-limited domination preserves the covering aspect of classical domination while simultaneously controlling the local load assigned to each dominating vertex.

Formally, let $G$ be a graph and let $k$ be a nonnegative integer. A set $D\subseteq V(G)$ is called a \emph{$k$-limited dominating set} if $D$ is a dominating set of $G$ and $|N(v)\setminus D|\le k$ for every $v\in D$. The minimum cardinality among all $k$-limited dominating sets of $G$ is called the \emph{$k$-limited domination number} of $G$ and is denoted by $\gamma_k^L(G)$. When $k\ge \Delta(G)$, the additional condition is trivially satisfied for every dominating set, and therefore $\gamma_k^L(G)=\gamma(G)$. On the other hand, if $k=0$, then every vertex of a $0$-limited dominating set has no neighbors outside the set, implying $\gamma_0^L(G)=|V(G)|$. Hence, for graphs with $\Delta(G)\ge 2$, the nontrivial range is $1\le k< \Delta(G)$.

The initial study in \cite{NAS} established several fundamental properties of the parameter $\gamma_k^L(G)$. In particular, general lower and upper bounds were derived, exact values were determined for several standard graph families, and connections with other domination-type parameters were investigated. Special attention was devoted to the case $k=1$, where the parameter was studied via its relation to the packing concept. This led to a characterization of graphs attaining the extreme value $\gamma_1^L(G)=n-1$, as well as to a structural description of graphs satisfying $\gamma_1^L(G)=\lceil \frac n2\rceil$.

In this paper, we continue this line of research and extend the study of $\gamma_k^L(G)$ to arbitrary values of $k$. After the preliminaries, we first focus on improving the lower bound of the parameter $\gamma_k^L(G)$. This is achieved via the concept of $k$-capacitated domination, which in particular enables us to determine the graphs satisfying $\gamma_k^L(G)=\lceil \frac{n}{k+1}\rceil$. We then turn to upper bounds in Section 4. In this part, the upper bound is refined for graphs satisfying $k<\delta(G)$, using the $k$-limited packing number, and we further characterize graphs attaining the extreme value $\gamma_k^L(G)=n-k$ under the corresponding degree assumption. Finally, Section 5 concludes the paper with a summary of the main results and some directions for further research.

\section{Preliminaries}

We begin by fixing standard notation and terminology. Let $G = (V(G),E(G))$ be a simple graph, where $V(G)$ and $E(G)$ are its vertex set and edge set, respectively. Two vertices $u,v \in V(G)$ are adjacent if $uv \in E(G)$. For a vertex $v \in V(G)$, the set of all vertices adjacent to $v$ is denoted by $N_G(v)$ and is called the \emph{open neighborhood} of $v$, while $N_G[v] = N_G(v)\cup\{v\}$ denotes its \emph{closed neighborhood}. When the graph $G$ is clear from the context, we omit the subscript. Further, the degree of a vertex $v$ is the number of neighbors of $v$, and it is denoted by $\deg_G(v)$; $\delta(G)$ and $\Delta(G)$ denote the minimum and maximum degree of $G$, respectively. A vertex of degree $1$ is called a \emph{pendant vertex}, and in a tree, such vertices are referred to as \emph{leaves}. For a set $S \subseteq V(G)$, we denote by $\overline{S}$ its complement, i.e.~$\overline{S}=V(G)\setminus S$, while $G[S]$ stands for the subgraph of $G$ induced by $S$.

\medskip
Recall that a set $D \subseteq V(G)$ is called a \emph{dominating set} of $G$ if every vertex in $\overline{D}$ has a neighbor in $D$. The minimum cardinality among all dominating sets is the \emph{domination number} $\gamma(G)$, and a dominating set of cardinality $\gamma(G)$ is called a $\gamma(G)$-set.

We now introduce the central notion of this paper. Let $k$ be a positive integer. A set $D \subseteq V(G)$ is called a \emph{$k$-limited dominating set} if it is a dominating set and satisfies
$$
|N(u)\setminus D| \le k, \qquad \mathrm{for\ every }\ u \in D.
$$
We refer to this additional condition as the \emph{$k$-limited constraint}. The minimum cardinality among all $k$-limited dominating sets is called the \emph{$k$-limited domination number} of $G$, denoted by $\gamma_k^L(G)$, and a minimum such set is called a $\gamma_k^L(G)$-set.

It is wort noting that if $G$ has connected components $G_1,\ldots,G_r$, then
$\gamma_k^L(G) = \sum_{i=1}^r \gamma_k^L(G_i)$,
and therefore we restrict ourselves to connected graphs of order at least three.

\medskip
In~\cite{NAS}, several basic properties of $\gamma_k^L(G)$ were established, including general bounds and exact values for some standard graph families. For completeness, we recall the results that will be used throughout the paper.

\begin{theorem}\cite{NAS}
\label{bounds}
Let $G$ be a connected graph of order $n$, and let $k$ be an integer such that $1\le k < \Delta(G)$. Then $\max\left\{\gamma(G),\ \lceil \frac{n}{k+1}\rceil\right\} 
\leq \gamma_k^L(G) \leq n-k$.
\end{theorem}

We also recall exact values of $\gamma_k^L(G)$ for several standard graph families. Throughout the paper, $P_n$, $C_n$, $S_n$, and $K_n$ denote the path, cycle, star, and complete graph on $n$ vertices, respectively. 

\begin{prop}\cite{NAS}\label{families}
Let $k$ be an integer. Then the following statements hold.
\begin{enumerate}
\item[$\mathrm{(i)}$] For every $n\ge 3$, $\gamma_1^L(P_n)=\left\lceil\frac n2\right\rceil$.
\item[$\mathrm{(ii)}$] For every $n\ge 3$, $\gamma_1^L(C_n)= \left \lfloor \frac n2 \right \rfloor+ \left \lceil \frac n4 \right \rceil - \left \lfloor \frac n4 \right \rfloor$.
\item[$\mathrm{(iii)}$] For every $n$ and $1\le k<n-1$, $\gamma_k^L(S_n)=n-k$.
\item[$\mathrm{(iv)}$] For every $n$ and $1\le k<n-1$, $\gamma_k^L(K_n)=n-k$.
\end{enumerate}
\end{prop}

It was also shown in~\cite{NAS} that graphs containing a universal vertex satisfy $\gamma_1^L(G)=n-1$. The following result extends this observation.

\begin{prop}\label{prop:k-universal}
Let $G$ be a connected graph of order $n$, and let $k$ be an integer such that $1\le k<\Delta(G)$. If $G$ contains $k$ universal vertices, then $\gamma_k^L(G)=n-k$.
\end{prop}

\begin{proof}
Let $U\subseteq V(G)$ be the set of all universal vertices of $G$, and let $D$ be a $\gamma_k^L(G)$-set. Suppose first that there exists an $u\in U\cap D$. Since $u$ is adjacent to every vertex of $G$, the $k$-limited constraint implies that $|\overline{D}|=|N(u)\setminus D|\le k$, and hence $|D|=n-|\overline{D}|\ge n-k$.

Now assume that $U\cap D=\emptyset$. Then all universal vertices belong to $\overline{D}$. If there exists a vertex $v \in \overline{D\cup U}$, i.e.~$v$ is not dominated by itself and it is not a universal vertex, then $v$ must be adjacent to some vertex $x \in D$. However, $x$ is then adjacent to $v$ and all vertices in $U$, implying $|N(x)\setminus D|\ge k+1$, a contradiction with the $k$-limited constraint. Hence, $\overline{D}=U$, and so, $|D|=n-|\overline{D}|=n-k$.

In both cases, we obtain $\gamma_k^L(G)\ge n-k$, and the conclusion then follows from Theorem~\ref{bounds}.
\end{proof}


\section{Lower bound via $k$-capacitated domination}
As already shown in \cite{NAS}, $k$-limited domination can be considered in various contexts related to other domination models, yet it remains a new concept. Among the domination models related to $k$-limited domination, perhaps the closest one is the notion of $k$-capacitated domination introduced in \cite{CD}.

Let $G=(V(G),E(G))$ be a finite simple graph. A set $S=\{v_1,v_2,\ldots, v_r\}\subseteq V(G)$ is called a \emph{$k$-capacitated dominating set} if the vertex set $V(G)$ admits a partition $(V_1,V_2,\ldots,V_r)$ such that for every $i\in\{1,2,\ldots,r\}$ we have
\begin{itemize}
\item[(i)] $v_i\in V_i$;
\item[(ii)] $|V_i|\le k+1$;
\item[(iii)] $v_i$ is adjacent to all other vertices in $V_i$.
\end{itemize}
The minimum cardinality among all such sets is called the $k$-capacitated domination number of $G$, denoted by $\tau_k(G)$. A minimum $k$-capacitated dominating set of $G$ is called a $\tau_k(G)$-set.

\medskip
As every $k$-capacitated dominating set is, in particular, a dominating set, it is clear that $\gamma(G)\le \tau_k(G)$ for all graphs $G$. Moreover, by \cite[Observation 1]{CD}, we also have $\lceil\frac{n}{k+1}\rceil\le\tau_k(G)$, showing that the lower bounds for $\tau_k(G)$ coincide with the lower bounds for $\gamma_k^L(G)$ given in Theorem~\ref{bounds}. This similarity naturally motivates a comparison between the two parameters. We also note that authors of \cite{HHH3} state that each vertex of a $k$-capacitated dominating set dominates at most $k$ vertices outside the set, which may suggest a resemblance to $k$-limited domination. The following examples show that the two concepts are nevertheless distinct.

Let us consider the graph $G$ obtained from the cycle $abca$ by adding a pendant vertex $d$ to $a$. For $k=1$, the set $S=\{a,c\}$ is a $1$-capacitated dominating set with partition $V(G)=(V_1, V_2)$, where $V_1=\{a,d\}$ and $V_2=\{b,c\}$. Hence, $\tau_1(G)\le|S|=2$. However, as a vertex $a$ is universal, Proposition \ref{prop:k-universal} ensures that $\gamma_1^L(G)=4-1=3$. Thus, $\tau_1(G)<\gamma_1^L(G)$.

Another simple case is the graph $H$ obtained from the previous graph $G$ by adding a vertex $u$ adjacent to both $a$ and $b$. Then $S=\{a,b\}$ with partition $(V_1,V_2)$ of $V(H)$, where $V_1=\{a,d\}$ and $V_2=\{b,c,u\}$, is a $2$-capacitated dominating set. Hence, $\tau_2(H)\le 2$. On the other hand, if $D$ is a $\gamma_2^L(H)$-set with $|D|\le2$, then $a\notin D$, as otherwise the $2$-limited constraint is violated. Hence, $d\in D$, and consequently, $|D|=2$. If $D=\{b,d\}$, then $|N(b)\setminus D|=3$, a contradiction, while in the cases $D=\{u,d\}$ or $D=\{c,d\}$ the vertex $c$ or $u$, respectively, is left undominated, a contradiction again. Therefore, $\gamma_2^L(H)\ge 3$ and so, $\tau_2(H)<\gamma_2^L(H)$.

These examples demonstrate that the concepts of $k$-limited and $k$-capacitated domination differ. The relationship between the corresponding parameters is now summarized in the following result.

\begin{theorem}\label{thm:CDS}
Let $G$ be a connected graph of order $n$, and let $k$ be an integer such that $1\le k<\Delta(G)$. If $D$ is a $\gamma_k^L(G)$-set, then $D$ is a $k$-capacitated dominating set as well. Consequently, 
\begin{center}
$\max\left\{\gamma(G),\lceil\frac{n}{k+1}\rceil\right\}
\le \tau_k(G)\le\gamma_k^L(G).$
\end{center}
\end{theorem}

\begin{proof}
In order to prove the first statement, let $D$ be a $\gamma_k^L(G)$-set. Since $D$ is a dominating set, every vertex $u \in \overline{D}$ has at least one neighbor in $D$, and for each such vertex $u$, choose arbitrarily one neighbor $f(u)\in D$. Then, for each $v \in D$, define
\begin{equation}\label{eq:CDS}
V_v := \{v\} \cup \{u \in \overline{D};\ f(u)=v\}.
\end{equation}
The sequence of sets $(V_v)_{v\in D}$ is a partition of $V(G)$, since each vertex in $D$ belongs to exactly one of these sets and every vertex in $\overline{D}$ is assigned to exactly one vertex in $D$.

We next verify that this partition satisfies properties (i)-(iii) from the definition of $k$-capacitated dominating set. Property (i) follows immediately by (\ref{eq:CDS}). For (ii) apply the fact that $D$ is a $k$-limited dominating set, and so $|N(v)\setminus D|\le k$, for every $v\in D$. Moreover, every vertex from $V_v\setminus\{v\}$ belongs to $N(v)\setminus D$. Hence,
$$
|V_v| \le 1 + |N(v)\setminus D| \le 1+k,
$$
giving the property (ii). Finally, every vertex in $V_v$ is either equal to $v$ or adjacent to $v$, so $v$ is adjacent to all other vertices in $V_v$. Therefore, the set $D$ with the partition $V(G)=(V_v)_{v\in D}$ is a $k$-capacitated dominating set of $G$.  

As $D$ was an arbitrary $\gamma_k^L(G)$-set, this shows that every $\gamma_k^L(G)$-set is a $k$-capacitated dominating set and so, $\tau_k(G)\le\gamma_k^L(G)$, completing the proof.
\end{proof}

The concept of $k$-capacitated domination actually gives a refined lower bound for the $k$-limited domination number. Since in~\cite{NAS} we posed the problem of characterizing graphs for which the bounds given in Theorem~\ref{bounds} are attained,  Theorem~\ref{thm:CDS} now provides additional insight into this question, allowing us to describe the structure of graphs attaining one of these extreme values. To do this, we need some additional notations.

\medskip
We say that a graph $G$ admits a \emph{$k$-star factor with $t$ components} if $G$ has a spanning subgraph with exactly $t$ connected components, each of which is either an isolated vertex or a star with at most $k$ leaves.

Furthermore, we say that $G$ admits a \emph{$k$-limited-star factor with $t$ components} if it admits a $k$-star factor with $t$ components such that every center of a star component (including an isolated vertex viewed as a trivial star) has at most $k$ neighbors outside the set of all centers of the factor. More precisely, if $C\subseteq V(G)$ denotes the set of centers of all star components in this factor, then $|N_G(v)\setminus C|\le k$ for every $v\in C$.

We remark that a closely related notion was introduced in \cite{gunther} under the name of a \emph{$k$-bounded star factor}. The only difference is that isolated vertices are not allowed as components there, whereas in our setting, they are viewed as stars without leaves.

\begin{theorem}\label{thm:kstar}
Let $G$ be a connected graph of order $n$, and let $k$ be an integer such that $1\le k<\Delta(G)$. Then, $\gamma_k^L(G)=\lceil\frac{n}{k+1}\rceil$ if and only if $G$ admits a $k$-limited-star factor with $\lceil\frac{n}{k+1}\rceil$ components.
\end{theorem}

\begin{proof}
Let $t:=\lceil\frac{n}{k+1}\rceil$ and assume first that $\gamma_k^L(G)=t$. Let $S=\{v_1,v_2,\ldots,v_t\}\subseteq V(G)$ be a $\gamma_k^L(G)$-set. By Theorem~\ref{thm:CDS}, $S$ is also a $k$-capacitated dominating set. Hence, there exists a partition $(V_1,V_2,\ldots,V_t)$ of $V(G)$ such that for every $i\in\{1,2,\ldots,t\}$ we have $v_i\in V_i$, $|V_i|\le k+1$, and $v_i$ is adjacent to (if they exist) all other vertices in $V_i$.

Now, define a spanning subgraph $F$ of $G$ whose edge set consists of all edges $v_i u$, where $u\in V_i\setminus\{v_i\}$ and $i\in\{1,2,\ldots,t\}$; we keep precisely those edges of $G$ that connect each $v_i$ to vertices in its corresponding part $V_i\setminus\{v_i\}$. Consequently, each component of $F$ is either an isolated vertex or a star with at most $k$ leaves. Thus, $F$ is a $k$-star factor with exactly $t$ components. Moreover, since $S$ is a $k$-limited dominating set of $G$, each vertex $v\in S$ has at most $k$ neighbors outside $S$. Therefore, $F$ is a $k$-limited-star factor with $t$ components.

\medskip
Conversely, assume that $G$ admits a $k$-limited-star factor with $t$ components. Let $F$ be such a factor, and let $C$ denote the set of centers of its components. Since each component of $F$ contributes exactly one vertex to $C$, we have $|C|=t$. In addition, every vertex of $\overline{C}$ belongs to a star component of $F$ and is therefore adjacent to its corresponding center in $C$. Hence, $C$ is a dominating set of $G$.

Furthermore, by the definition of a $k$-limited-star factor, every vertex in $C$ has at most $k$ neighbors outside $C$. Therefore, $C$ is a $k$-limited dominating set of $G$, and consequently $\gamma_k^L(G)\le |C|=t$. Together with the general lower bound from Theorem~\ref{bounds}, we conclude that $\gamma_k^L(G)=t$.
\end{proof}

The special case $k=1$ admits a particularly simple interpretation. In this situation, a $1$-limited-star factor with $\lceil\frac n2\rceil$ components consists of pairwise disjoint copies of $K_2$ together with, possibly, a single isolated vertex. Hence, the existence of such a factor is equivalent to the existence of a perfect matching when $n$ is even, and a near-perfect matching when $n$ is odd. Therefore, Theorem~\ref{thm:kstar} may be viewed as a natural extension of Corollary~14 from \cite{NAS} to arbitrary values of $k$.

\medskip
To close this section, we observe that $\gamma_k^L(G)\ge 2$ for every connected graph $G$ of order $n\ge 3$. Indeed, if $\gamma_k^L(G)=1$, then Theorem~\ref{bounds} yields $\lceil \frac{n}{k+1} \rceil =1$, and hence $n\le k+1$. However, since $\Delta(G)\le n-1$, and thus $k<\Delta(G)\le n-1\le k$, a contradiction arises. Therefore, the smallest possible value of $\gamma_k^L(G)$ is $2$. It is then natural to characterize graphs attaining this bound. Although the order of such graphs cannot exceed $2k+2$, Theorem~\ref{thm:kstar} shows that there exist numerous non-isomorphic graphs with this property. Consequently, rather than pursuing a general characterization of graphs $G$ with $\gamma_k^L(G)=2$, we restrict our attention to two classical graph families.

For this purpose, we say that an edge $xy\in E(G)$ is a \emph{$k$-limited dominating edge} if $\{x,y\}$ is a dominating set with $\deg_G(x)\le k+1$ and $\deg_G(y)\le k+1$.

\begin{cor}\label{cor:all}
Let $G$ be a connected graph of order $n$, and let $k$ be an integer such that $1\le k<\Delta(G)$. Then the following statements hold.
\begin{itemize}
\item[$\mathrm{(i)}$] $\gamma_1^L(G)=2$ if and only if $G\in\{P_3,P_4,C_3,C_4\}$.
\item[$\mathrm{(ii)}$] If $G$ is a tree and $k\ge 2$, then $\gamma_k^L(G)=2$ if and only if $G$ is a double star $S_{k,q}$ of order $k+q+2$ for some $q\le k$.
\item[$\mathrm{(iii)}$] If $G$ is a $d$-regular graph, then $\gamma_k^L(G)=2$ if and only if $d=k+1$ and $G$ admits a $k$-limited dominating edge.
\end{itemize}
\end{cor}

\begin{proof}
(i) For sufficiency, apply Proposition \ref{families}, and for necessity, suppose that $\gamma_1^L(G)=2$. Since each vertex in $\gamma_1^L(G)$-set can have at most one neighbor outside $D$, we obtain $n\le 4$. If  $n=3$, Theorem \ref{thm:kstar} ensures that $G\in\{P_3,C_3\}$, while for $n=4$ provides that $G\in\{P_4,C_4\}$.

\medskip
(ii) If $G=S_{k,q}$ with $q\le k$, then its two central vertices form a dominating set of cardinality two, and each has at most $k$ neighbors outside this set. Hence $\gamma_k^L(G)=2$. 

Conversely, assume that $\gamma_k^L(G)=2$. Since $2=\lceil\frac{n}{k+1}\rceil$, Theorem~\ref{thm:kstar} implies that $G$ admits a $k$-limited-star factor with exactly two components, say $F_1$ and $F_2$. As $G$ is a tree, these two components must be connected by exactly one edge. Moreover, since $k<\Delta(G)$, one of the components, say $F_1$, is a star with exactly $k$ leaves; denote its center by $c_1$. Consequently, as $\Delta(F_1)=k$ and $\Delta(F_2)\le k$, the edge joining $F_1$ and $F_2$ must be incident with $c_1$, and therefore $\deg_G(c_1)=k+1$. In turn, this edge is a $2$-limited dominating edge of $G$. Finally, the endpoint of this edge lying in $F_2$ must dominate all vertices of $F_2$, and hence it is the center of $F_2$ (or possibly the isolated vertex if $F_2$ is trivial). Therefore, $G$ is a double star $S_{k,q}$ for some $q\le k$.

\medskip
(iii) Let $G$ be a $d$-regular graph with $k=d-1$. If $G$ admits a $k$-limited dominating edge $xy$, then $\{x,y\}$ is a dominating set of $G$, and each of $x$ and $y$ has at most $k$ neighbors outside $\{x,y\}$. Hence $\{x,y\}$ is a $k$-limited dominating set, and thus $\gamma_k^L(G)=2$.

For the opposite implication, assume that $\gamma_k^L(G)=2$ and let $D=\{u,v\}$ be a $\gamma_k^L(G)$-set. Then every vertex in $\overline{D}$ is adjacent to at least one of $u$ or $v$. In particular, $u$ and $v$ must be adjacent; otherwise, all neighbors of $u$ belong to $\overline{D}$, so $|N(u)\setminus D|=|N(u)|=d>k$, a contradiction with the $k$-limited constraint. Therefore, $uv$ is a $k$-limited dominating edge. Next, as $u$ has exactly $d-1$ neighbors outside $D$, the $k$-limited constraint implies that $k\ge d-1$. Together with $k<d$, we obtain $d=k+1$.
\end{proof}

\section{Upper bounds via $k$-limited packing}

We now turn to the upper bound of the $k$-limited domination number. By Theorem~\ref{bounds}, we have $\gamma_k^L(G) \le n-k$ for every graph $G$. A natural question is to characterize graphs attaining this bound. By Proposition~\ref{families}, this equality holds in particular for complete graphs and star graphs, which raises the question of whether other graph families share this extremal property. The answer is already affirmative by Proposition \ref{prop:k-universal}, so we want to know what the structure of such graphs is.

This problem has already been studied in \cite{NAS} for $k=1$. In this work, we extend the investigation to $k$ with $1\le k\le\delta(G)$. To this end, we first introduce the following notions, adopted by \cite{Mojdeh, Babak}.

\medskip
A subset $B \subseteq V(G)$ is called a \emph{packing} in a simple graph $G=(V(G),E(G))$ if for every pair of distinct vertices $u,v \in B$, we have $N[u]\cap N[v]=\emptyset$. The maximum cardinality among all packings in $G$ is called the \emph{packing number}, denoted by $\rho(G)$. Equivalently, a set $B \subseteq V(G)$ is a packing if and only if $|N[u]\cap B| \le 1$ for every $u \in V(G)$.

More generally, a subset $B \subseteq V(G)$ is called a \emph{$k$-limited packing} if $|N[u]\cap B| \le k$ for every $u \in V(G)$. The maximum cardinality among all $k$-limited packings in $G$ is called the \emph{$k$-limited packing number}, denoted by $L_k(G)$.

In the following lemma, we show that under a degree condition, every $k$-limited packing naturally yields a $k$-limited dominating set in $G$.

\begin{lemma}\label{lem}
Let $G$ be a connected graph of order $n$, and let $k$ be an integer such that $1\le k\le \delta(G)$. If $B$ is a $L_k(G)$-set, then $\overline{B}$ is a $k$-limited dominating set of $G$.
\end{lemma}

\begin{proof}
Let $B$ be an $L_k(G)$-set. For every $u \in B$ we have $|N[u]\cap B|\le k$. Hence $|N(u)\cap B|\le k-1$, so $u$ has at most $k-1$ neighbors in $B$. Since $\deg_G(u)\ge k$, it has at least one neighbor in $\overline{B}$. Therefore, $\overline{B}$ is a dominating set of $G$.

Next, let $v \in \overline{B}$. As $v \notin B$, it follows that
$|N(v)\cap B| = |N[v]\cap B| \le k$, when applying that $B$ is a $k$-limited packing. Consequently, $|N(v)\setminus \overline{B}| = |N(v)\cap B| \le k$, and hence every vertex in $\overline{B}$ has at most $k$ neighbors outside $\overline{B}$, giving that $\overline{B}$ is a $k$-limited dominating set of $G$, as we claimed.
\end{proof}

The inequality $\rho(G)+k-1 \le L_k(G)$ from \cite[Theorem~3]{Mojdeh}, valid for every connected graph $G$ of order $n$ and every $k \le \Delta(G)$, yields
$n - L_k(G) \le n - k - (\rho(G)-1) \le n - k$. Hence, it provides an improvement over the general upper bound for $\gamma_k^L(G)$ given in Theorem~\ref{bounds}.

\begin{theorem} \label{k-packing}
Let $G$ be a connected graph of order $n$, and let $k$ be an integer such that $1\le k\le \delta(G)$. Then $\gamma_k^L(G) \le n - L_k(G)$. Moreover, equality holds if and only if there exists an $L_k(G)$-set $B$ such that $\overline{B}$ is a $\gamma_k^L(G)$-set.
\end{theorem}

\begin{proof}
If $B$ is an $L_k(G)$-set, then by Lemma~\ref{lem}, $\overline{B}$ is a $k$-limited dominating set of $G$, and hence $\gamma_k^L(G)\le |\overline{B}|=n-L_k(G)$.

For the second claim, let $B$ be an $L_k(G)$-set and $\overline{B}$ a $\gamma_k^L(G)$-set. Thus, $\gamma_k^L(G)=|\overline{B}|=n-|B|=n-L_k(G)$. Conversely, assume that $\gamma_k^L(G)=n-L_k(G)$ and let $B$ be an $L_k(G)$-set. By Lemma \ref{lem} we infer that $n-L_k(G)=\gamma_k^L(G)\le |\overline{B}|=n-|B|=n-L_k(G)$. Therefore, $\gamma_k^L(G)=|\overline{B}|$ and $\overline{B}$ is a $\gamma_k^L(G)$-set.
\end{proof}

Let us remark that the assumption $k\le\delta(G)$ in Theorem~\ref{k-packing} is essential; the upper bound does not hold in general when $k>\delta(G)$. For instance, let us consider the tree $T$ on six vertices with $V(T)=\{a,b,c,d,e,f\}$ and edge set $E(T)=\{ab,bc,bd,ce,df\}$. Observe that $\delta(T)=1$ and $\Delta(T)=3$, so let us take $k=2$. It is then clear that the set $B=\{a,b,e,f\}$ is a $2$-limited packing of $T$, implying that $L_2(T)\ge4$. This is further equivalent to $n-L_2(T)\le 2$. On the other hand, we claim that $\gamma_2^L(T)\ge 3$. If this is not the case, then $\gamma_2^L(T)=2$ by Observation preceding Corollary \ref{cor:all}, and so, there exists a $2$-limited dominating set $D\subseteq V(T)$ with $|D|=2$. Since $D$ is actually a dominating set, the vertex $b$ must belong to $D$, as otherwise the leaves $a,e$ and $f$ cannot all be dominated by two vertices. Then, since $\deg_T(b)=3$, at least one of its neighbors belongs to $D$. As $|D|=2$, it is clear that either $D=\{b,c\}$ or $D=\{b,d\}$. In the first case, the vertex $f$ is left undominated, while in the second case the vertex $e$ is not dominated, a contradiction. Therefore, $\gamma_2^L(T)\ge 3$, and in turn  $\gamma_2^L(T)> n-L_2(T)$. 
This shows that Theorem~\ref{k-packing} can not be extended to the case $k>\delta(G)$. So, if the condition $k\le\delta(G)$ is removed from the assumption of Theorem \ref{k-packing}, we obtain the following weaker upper bound.

\begin{prop}\label{packingU}
Let $G$ be a connected graph of order $n$, and let $k$ be an integer such that $2\le k<\Delta(G)$. If $U_k=\{u\in V(G);\ \deg_G(u)< k\}$,
then
$$
\gamma^L_k(G)\le n-L_k(G)+|U_k|.
$$
\end{prop}

\begin{proof}
Let $B$ be an $L_k(G)$-set and observe that $U_k\subseteq B$. We claim that $D:=\overline{B}\cup U_k$ is a $k$-limited dominating set of $G$.

To see that $D$ is a dominating set, let $v\in \overline{D}=B\setminus U_k$. Applying that the set $B$ is a $k$-limited packing gives that  $|N[v]\cap B|\le k$, i.e. $|N(v)\cap B|\le k-1$. On the other hand, as $v\notin U_k$, it follows that $\deg_G(v)\ge k$. Hence, $v$ has at least one neighbor in $D$.

It remains to be verified that the $k$-limited condition holds. For every vertex of $U_k$ this is immediate, so let $u\in\overline{B}$. Then $|N(u)\setminus D|\le |N(u)\setminus\overline{B}|
=|N[u]\cap B|\le k$.  Thus, every vertex of $D$ has at most $k$ neighbors outside $D$, and therefore $D$ is a $k$-limited dominating set.

Consequently, as $U_k\subseteq B$, we obtain $\gamma_k^L(G)\le |D|
=|\overline{B}|+|U_k|=n-L_k(G)+|U_k|$.
\end{proof}

In Theorem~\ref{k-packing}, the degree assumption guarantees an improvement over the general upper bound for $\gamma_k^L(G)$ from Theorem~\ref{bounds}, whereas this is no longer true for Proposition~\ref{packingU}. To compare the bound from
Proposition~\ref{packingU} with the general bound $n-k$, observe that
$$
n-L_k(G)+|U_k|\le n-k
\quad \iff \quad
L_k(G)\ge k+|U_k|.
$$
Hence, Proposition~\ref{packingU} yields a stronger upper bound for
$\gamma_k^L(G)$ whenever $L_k(G)\ge k+|U_k|$. As the next example
demonstrates, this inequality may hold for some values of $k$ and fail for others, even within the same graph.

For each $i \in \{1,2,3\}$, let $G_i$ be the star graph with center $c_i$ and leaves $a_i,b_i,d_i$. Further, let $x,y$ be two additional vertices forming a path $xy$. The graph $G$ is then obtained from the disjoint union of $G_1,G_2,G_3$ and $xy$ by adding only the edges $xa_i$ and $yb_i$ for all $i\in\{1,2,3\}$. Hence, $\Delta(G)=4$, so we consider $k\in\{2,3\}$.

Since $\{a_1,a_2,a_3,b_1,b_2,b_3\}$ is a $2$-limited packing, we have $L_2(G)\ge 6$. Moreover, as $|U_2|=3$, it follows that $L_2(G) \ge 2+|U_2|$, so Proposition~\ref{packingU} yields a better bound for $\gamma_k^L(G)$ than the general bound $n-k$.

On the other hand, each subgraph $G_i$ contributes at most three vertices to any $3$-limited packing of $G$, since including $c_i$ together with all its neighbors in a $3$-limited packing set $B$ would imply that $|N[c_i]\cap B|=4$, a contradiction. Hence, $L_3(G)\le 9+2=11$, and consequently $L_3(G) < 3 + |U_3| = 12$. Thus, in this case, the general bound for $\gamma_k^L(G)$ is stronger than the bound from Proposition~\ref{packingU}.

\medskip
It is worth noting that when $k\le\delta(G)$, the refinement of the upper bound from Theorem~\ref{bounds} enables us to characterize graphs attaining the extreme value $\gamma_k^L(G)=n-k$. For this purpose, we need to generalize the notion of external private neighborhood from \cite{HHH1}. In the general setting, for a set $T \subseteq V(G)$ and a vertex $t \in T$, the \emph{external $k$-limited neighborhood of $t$ with respect to $T$} is defined as
$$
\mathrm{eln}_k(t,T) = \{ u \in \overline{T};\ |N(u)\cap T| \le k \ \mathrm{and}\ t \in N(u) \}.
$$

\begin{theorem}\label{thm:n-k}
Let $G$ be a connected graph of order $n$, and let $k$ be an integer such that $1\le k<\delta(G)$. Then $\gamma_k^L(G)=n-k$ if and only if $G$ contains $k$ universal vertices or both of the following statements hold.
\begin{enumerate}
\item[$\mathrm{(i)}$] For every subset $S \subseteq V(G)$ with $|S|=k+1$,  the vertices of $S$ have a common neighbor in $\overline{S}$.

\item[$\mathrm{(ii)}$] There is no subset $T \subseteq V(G)$ with $|T|\ge k+2$ such that $\mathrm{eln}_k(t,T) \neq \emptyset$ for every $t \in T$ and $N(T) = \bigcup_{t \in T} \mathrm{eln}_k(t,T)$.
\end{enumerate}
\end{theorem}

\begin{proof}
Assume first that $\gamma_k^L(G)=n-k$ and that $G$ does not contain $k$ universal vertices. To establish property~(i), suppose to the contrary that there exists a set $S\subseteq V(G)$ with $|S|=k+1$ whose vertices do not have a common neighbor in $\overline{S}$. Since $\delta(G)\ge k+1$, every vertex of $S$ has a neighbor in $\overline{S}$, and therefore $\overline{S}$ is a dominating set of $G$. Moreover, as no vertex of $\overline{S}$ is adjacent to all vertices of $S$, each vertex of $\overline{S}$ has at most $k$ neighbors outside $\overline{S}$. Thus $\overline{S}$ is a $k$-limited dominating set. Consequently, $\gamma_k^L(G)\le|\overline{S}|=n-(k+1)$, contradicting the assumption that $\gamma_k^L(G)=n-k$. Hence, property~(i) holds.

Next, we verify property~(ii). Suppose, for a contradiction, that there exists a subset $T\subseteq V(G)$ with $|T|\ge k+2$ such that
$$
N(T)=\bigcup_{t\in T}\mathrm{eln}_k(t,T)
\qquad\mathrm{and}\qquad
\mathrm{eln}_k(t,T)\neq\emptyset
\quad\mathrm{for\ every }\ t\in T.
$$
The latter property implies that every vertex of $T$ has a neighbor in $\overline{T}$, so $\overline{T}$ is a dominating set of $G$. Furthermore, by the former property, every vertex $u\in N(T)$ belongs to $\mathrm{eln}_k(t,T)$ for some $t\in T$, and thus $|N(u)\cap T|\le k$. More precisely, every vertex of $\overline{T}$ has at most $k$ neighbors outside $\overline{T}$. Therefore, $\overline{T}$ is a $k$-limited dominating set. Since $|T|\ge k+2$, we obtain $\gamma_k^L(G)\le |\overline{T}|=n-|T|\le n-(k+2)$, again contradicting $\gamma_k^L(G)=n-k$. This proves property~(ii).

\medskip
Conversely, if $G$ contains $k$ universal vertices, then $\gamma_k^L(G)=n-k$ follows directly from Proposition~\ref{prop:k-universal}. Hence, we may assume that $G$ satisfies properties~(i) and~(ii) and contains no set of $k$ universal vertices. Since $n\ge k+2$, property~(i) provides that $L_k(G)=k$ by~\cite[Theorem 3.4]{BCL}. Let $B$ be an $L_k(G)$-set. By Lemma~\ref{lem}, the set $\overline{B}$ is then a $k$-limited dominating set of $G$ with $|\overline{B}|=n-k$. We claim that it is in fact a $\gamma_k^L(G)$-set.

Suppose otherwise. Then there exists a $\gamma_k^L(G)$-set $D$ with
$|D|\le n-(k+1)$, and therefore $|\overline{D}|\ge k+1$. We first show that $|\overline{D}|=k+1$. Indeed, if $|\overline{D}|\ge k+2$, then every vertex $v\in\overline{D}$ has a neighbor in $D$ (as $D$ is a dominating set of $G$). Moreover, $k$-limited constraint ensures that $\mathrm{eln}_k(v,\overline{D})=N(v)\cap D\neq\emptyset$ for all $v\in\overline{D}$. In turn, as
$$
N(\overline{D})
=\bigcup_{v\in\overline{D}}(N(v)\cap D)
=\bigcup_{v\in\overline{D}}\mathrm{eln}_k(v,\overline{D}),
$$
we are in a contradiction with property~(ii). Therefore, $|\overline{D}|=k+1$. Applying property~(i) to the set $\overline{D}$, we conclude that its vertices have a common neighbor in $D$. Such a vertex would be adjacent to all $k+1$ vertices of $\overline{D}$, contradicting the $k$-limited constraint. 

This shows that no $k$-limited dominating set of cardinality less than $n-k$ exists. Accordingly, $\overline{B}$ is a $\gamma_k^L(G)$-set. Finally, Theorem~\ref{k-packing} provides that $\gamma_k^L(G)=|\overline{B}|=n-k$.
\end{proof}

Let us emphasize that there exist graphs $G$ with $\gamma_k^L(G)=n-k$ that do not satisfy the properties of Theorem~\ref{thm:n-k}; for instance, by Proposition~\ref{families}(iii), stars satisfy $\gamma_k^L(S_n)=n-k$, although property~(i) fails. Consequently, Theorem~\ref{thm:n-k} does not characterize all graphs attaining the extreme value $\gamma_k^L(G)=n-k$, nor does it characterize the subclass of such graphs with $\delta(G)<k$. On the other hand, since every regular graph satisfies $\delta(G)=\Delta(G)$, Theorem~\ref{thm:n-k} yields a complete characterization of regular graphs with $\gamma_k^L(G)=n-k$.

\section{Conclusion}

In this paper, we continued the study of $k$-limited domination initiated in~\cite{NAS}. We improved the general lower bound on $\gamma_k^L(G)$ by relating the parameter to $k$-capacitated domination, and we refined, under a degree condition, the general upper bound by employing the $k$-limited packing number. In particular, when $k\le \delta(G)$, combining known results with those obtained here yields
$$
\rho(G)\le \gamma(G)\le \tau_k(G)\le \gamma_k^L(G)\le n-L_k(G), 
$$
giving the Gallai-type inequalities $\tau_k(G)+L_k(G)\le n$ and $\gamma_k^L(G)+L_k(G)\le n$.

The connection with k-capacitated domination appears particularly promising. While capacitated domination has mainly been studied from algorithmic and computational complexity perspectives (see, for example, \cite{Becker,Kao, LTV}), much less is known about the parameter $\tau_k(G)$ itself. Since $\tau_k(G)\le\gamma_k^L(G)$, the bounds obtained in this paper provide a partial answer to a question raised in \cite{CD} concerning meaningful upper bounds for $\tau_k(G)$, and suggest that $k$-limited domination may serve as a useful tool in the study of capacitated domination.

Several natural questions remain open. The parameter $\gamma_k^L(G)$ is naturally bounded below by both $\gamma(G)$ and $\lceil\frac{n}{k+1}\rceil$. While one of the main results of this paper provides a characterization of graphs attaining the bound $\gamma_k^L(G)=\lceil\frac{n}{k+1}\rceil$, the question of when the other lower bound is attained remains open. In \cite{NAS}, it was shown that the equality $\gamma_k^L(G)=\gamma(G)$ can occur, in particular for efficient graphs under some additional assumption. This naturally leads to the following question.

\begin{prb}
Can we characterize the graphs $G$ for which $\gamma_k^L(G)=\gamma(G)$?
\end{prb}

While the previous question concerns the smallest possible value of $\gamma_k^L(G)$, it is also natural to investigate the opposite extreme case. In this paper, we characterize graphs satisfying $\gamma_k^L(G)=n-k$ under the assumption $k< \delta(G)$. However, the structure of such graphs remains unknown when $\delta(G)\le k<\Delta(G)$.

\begin{prb}
Can we characterize the graphs $G$ satisfying $\gamma_k^L(G)=n-k$ when $\delta(G)\le k<\Delta(G)$?
\end{prb}

\vskip 1pc \noindent{\bf Acknowledgments.} 
The researcher was partially supported by the Slovenian research agency ARIS, program no.\ P1-0288.

\bibliographystyle{plain}
\bibliography{G-Arxiv}

\end{document}